\documentclass[final,leqno,onefignum,onetabnum]{siamltex1213}

\usepackage{amssymb}
\usepackage{amsfonts}
 \usepackage{graphicx}
 \usepackage{tikz}
\usetikzlibrary{shapes,backgrounds}
\usepackage{verbatim}
\usepackage[normalem]{ulem}
\usepackage{hyperref}
\usepackage{url}

\newcommand{\y}{\mathbf}

\newcommand{\bZ}{\mathbb Z}

\newcommand{\bR}{\mathbb R}

\newtheorem{op}{Open Problem}

\makeatletter
\@addtoreset{case}{thm}
\@addtoreset{case}{lem}
\@addtoreset{case}{prop}
\@addtoreset{case}{cor}
\@addtoreset{case}{section}
\makeatother

\newtheorem{exmp}{Example}[section]

\title{H-Representation of the {Kimura-3} Polytope} 

\author{Marie Mauhar\footnotemark[2]\
\and Joseph Rusinko\footnotemark[3]\
\and Zoe Vernon\footnotemark[4]\ }

\begin{document}
\maketitle

\renewcommand{\thefootnote}{\fnsymbol{footnote}}

\footnotetext[2]{Lenoir-Rhyne University }
\footnotetext[3]{Hobart and William Smith Colleges, email: rusinko\@hws.edu}
\footnotetext[4]{Washington University}

\renewcommand{\thefootnote}{\arabic{footnote}}

\slugger{mms}{xxxx}{xx}{x}{x--x}

\begin{abstract}
Given a group-based Markov model on a tree, one can compute the vertex representation of a polytope describing a toric variety associated to the algebraic statistical model. In the case of $\bZ_2$ or $\bZ_2\times\bZ_2$, these polytopes have applications in the field of phylogenetics. We provide a half-space representation for the $m$-claw tree where $G=\bZ_2\times\bZ_2$, which corresponds to the {Kimura-3} model of evolution. 
\end{abstract}

\begin{keywords} Facet description, Polytopes, Group-based phylogenetic models \end{keywords}

\begin{AMS}	52B20 \end{AMS}

\pagestyle{myheadings}
\thispagestyle{plain}
\markboth{Kimura-3 Polytope}{}

\section{Introduction}
\label{intro}
Phylogenetic trees depict evolutionary relationships between proteins, genes or organisms. Many tree reconstruction methods assume evolution is described as a Markov process along the edges of the tree. The Markov matrices define the probability that a characteristic changes along the edge of the tree. The probability of observing a particular collection of characteristics at the leaves of the tree can be computed as a polynomial in the (unknown) entries of the transition matrices. See \protect\cite{geometryk3anytree} for an overview of this algebraic statistical viewpoint on phylogenetics.

Given a tree and evolutionary model, invariants are polynomial relationships satisfied by the expected pattern frequencies occurring in sequences evolving along the tree under the Markov model \protect\cite{cavender,evans}. Algebraic geometry provides a framework for computing the complete set of invariants as the elements of a prime ideal that define an algebraic variety. Many classical varieties arise in the study of phylogenetic models \protect\cite{phyloAG} as do modern objects such as conformal blocks \protect\cite{conformal} and Berenstein-Zelevinsky triangles \protect\cite{cblocksandBZ}. 

For some models of evolution, known as \emph{group-based models}, there is a finite group which acts freely and transitively on the set of states in the transition matrix. Two such models are the {Kimura-3} model ($G=\bZ_2\times\bZ_2$), which accounts for differences in DNA mutation rates between transitions and both types of transversions \protect\cite{kimura}, as well as the binary symmetric model ($G=\bZ_2)$. In a group-based model the transition matrices can be simultaneously diagonalized \protect\cite{evans}. After the change of coordinates induced by these diagonalizations, the variety is seen to be toric \protect\cite{pachtersturmfels, toricideal}. As a toric variety, there is an associated polytope whose combinatorial structure describes the geometry of the variety. We refer the reader to \protect\cite{cox,fulton} for background on toric varieties and polytopes.

For a fixed tree and finite group there is an algorithm for computing the vertices of this polytope  (i.e. the vertex or \emph{V-representation}) \protect\cite{local,dbm}. There is an equivalent description of the polytope as the collection of points satisfying a set of linear inequalities, known as the half-space or \emph{H-representation}.  Translating between the vertex and half-space description of a polytope is an NP-complete problem in general \protect\cite{verticesnphard} and is challenging even in classes of examples where a plausible set of facets can be proposed. Subsequently, the H-representation for group-based models is known only in the case of the binary symmetric model \protect\cite{binary}. 

Finding an H-representation is difficult even for simple classes of trees. Such a representation is unknown in the case of a tree with one interior node and $m$ leaves, a tree referred to as the \emph{$m$-claw tree}. Claw trees play an important role in phylogenetics, as the variety associated with any tree can be computed by a sequence of toric fiber products of the varieties associated to claw trees \protect\cite{toricideal,toricfiber}. 

In this article we provide an H-representation associated to the claw tree of the {Kimura-3} model. This description of the {Kimura-3} polytope builds off of a well-known identification of the polytope for the binary symmetric model with the demihypercube. This work complements the growing body of knowledge about the geometry of the {Kimura-3} variety \protect\cite{geometryofk3,hilbert,geometryk3anytree}.

\subsection*{Outline of Article}
\label{subsec:outline}
In \S \protect\ref{sec:background} we review the vertex description of the polytope associated with the binary symmetric and {Kimura-3} models. We introduce a polytope $\Delta(m)$ that we propose as the H-representation of the polytope associated with the {Kimura-3} model for an arbitrary claw tree and show that if $\Delta(m)$ is integral, then it is the H-representation. In \S \protect\ref{sec:alternative} we introduce an isomorphism between $\Delta(m)$ and a polytope $\Delta'(m)$ described in terms of a $3\times m$ matrix whose row and column coordinates satisfy a set of inequalities reminiscent of those that define the demihypercube. We then identify a connection between the number of integral coordinates of a point and the number of facets on which it lies. This connection is utilized in \S \protect\ref{sec:integrality} to prove the main result of our paper, that $\Delta(m)$ is an integral polytope, and thus provides an H-representation of the {Kimura-3} polytope associated with a claw tree. We conclude with a collection of open problems about the combinatorics and geometry of group-based models.

\section{Polytopes for Group-Based Models}
\label{sec:background}
For a fixed tree and group-based model with group $G$, the V-representation of the polytope can be computed using an algorithm described in \protect\cite{local,dbm}. For each leaf, one defines a map $g:G \rightarrow \bR^{|G|-1}$ where the identity maps to $(0,0,\cdots,0)$ and each non-identity element maps to a standard basis vector of the integral lattice $\bZ^{|G|-1}$. Under this identification, the polytope associated to the $m$-claw tree is the convex hull in $\bR^{m(|G|-1)}$ of all possible labellings of the leaves with $m$ group elements that sum to the identity. Throughout this paper when discussing claw trees we assume that the number of leaves is at least three. We view $\bR^{m(|G|-1)}$ as entries of a $|G|-1 \times m$ matrix $M$ whose columns are indexed by the $m$ leaves of the tree.

\begin{definition}[\protect\cite{toricideal} as described in \protect\cite{dbm}]
\label{thm:K3vert} The Kimura polytope, denoted $K(m)$ is the polytope associated to the $m$-claw tree with the group $\bZ_2 \times \bZ_2$. The vertices of $K(m) \subseteq \bR^{3m}$ are in bijection with collections of $m$ elements of $\bZ_2 \times \bZ_2$ such that the sum of these elements is the identity.  We identify the of elements of $\bZ_2 \times \bZ_2$ with column vectors in a $3 \times m$ matrix $M$ under the map $g:\bZ_2 \times \bZ_2 \rightarrow \bR^3$ given by $g(0,0)=(0,0,0),   g(1,0)=(1,0,0),g(0,1)=(0,1,0)$, and $g(1,1)=(0,0,1)$.
\end{definition}

Our construction of a facet description for the $K(m)$ was inspired by the H-representation for the polytope associated to the group $\bZ_2$. In the case of $\bZ_2$ the polytope is the demihypercube which has a well known H-representation. 
\begin{exmp}[See \protect\cite{binary}]
The H-representation of the demihypercube is given by \[
DH(m)=\{d\in [0,1]^m:\sum_{i\in A}d_{i}\leq| A|-1+\sum_{j\notin A}d_{j}\},
\] where $A$ ranges over all subsets of $\{1,2,\cdots, m\}$ of odd cardinality.
\end{exmp}

To connect the two polytopes, we note that an element of $\bZ_2 \times \bZ_2$ is the identity if, and only if, its image is the identity under all homomorphisms from $\bZ_2 \times \bZ_2$ to $\bZ_2$. Therefore the sum of the images of all of the group elements defining a vertex of the Kimura polytope must be the identity under the three non-trivial homomorphisms from $\bZ_2 \times \bZ_2$ to $\bZ_2$. We confirm this relationship in Theorem~\protect\ref{thm:Deltaintegral}, where we show the following inequalities provide an H-representation for the {Kimura-3} polytope.

\begin{definition} 
\label{def:delta}
The polytope $\Delta(m)$ is the set of points in $\bR^{3m}$ satisfying $ x_{ij} \ge 0$ for all $1 \le i \le 3$ and $ 1 \le j \le m$ and $ \displaystyle\sum_{i=1}^3 x_{ij} \le 1$ for all $ 1\le j \le m$, as well as the following collection of \emph{$A$-inequalities}:
\begin{eqnarray*}
\sum_{j \in A}( x_{1j} + x_{2j}) &\le& |A|-1+ \sum_{l \not\in A}(x_{1l} + x_{2l}), \\
\sum_{j \in A}( x_{1j} + x_{3j}) &\le& |A|-1+ \sum_{l \not\in A}(x_{1l} + x_{3l})\mbox{, and } \\
\sum_{j \in A}( x_{2j} + x_{3j}) &\le& |A|-1+ \sum_{l \not\in A}(x_{2l} + x_{3l}).
\end{eqnarray*}
Where $A$ ranges over all odd cardinality subsets of $\{1,2,3,\cdots, m\}$.
\end{definition}

We first show that $\Delta(m)$ contains the {Kimura-3} polytope.

\begin{theorem}
\label{thm:KinDelta}
The polytope $K(m)\subseteq \Delta(m)$.
\end{theorem}
\begin{proof}
Let $v$ be a vertex of $K(m)$ corresponding to a choice $(g_1,g_2,\cdots, g_m)$ of elements of $\bZ_2 \times \bZ_2$ which sum to the identity. By the identification in Definition~\protect\ref{thm:K3vert}, $v$ satisfies $x_{ij} \ge 0$ and $ \sum_{i=1}^3 x_{ij} \le 1$ for all $1 \le j \le m$. 

Without loss of generality, we prove that $v$ satisfies the A-inequality \[\sum_{j \in A}( x_{1j} + x_{2j}) \le |A|-1+ \sum_{j \not\in A}(x_{1j} + x_{2j}).\]

Letting $X=\{k\in A |g_k = (1,0)\}$ and $Y=\{k\in A |g_k = (0,1)\}$, we have ${|X|+|Y| \le |A|}$. We divide the proof into two cases according to the parity of $|X|$ and $|Y|$. 

If $ |X|$ and $|Y|$ are of the same parity, then $|X|+|Y|\neq |A|$, since $|A|$ is odd. Therefore,
\begin{eqnarray*}
\sum_{j \in A}( x_{1j} + x_{2j})& = & |X|+|Y| \\ 
  & \le & |A|-1 \\ 
 & \le &|A|-1+ \sum_{j \not\in A}(x_{1j} + x_{2j}).
\end{eqnarray*}

If $ |X|$ and $|Y|$ are of opposite parity, then, without loss of generality, we assume $|X|$ is odd. This implies \[\sum_{i \in X \cup Y} g_i = (1,0). \] 

In order to neutralize $(1,0)$ in $\bZ_2 \times \bZ_2$, we must either add $(1,0)$, or both $(1,1) $ and $(0,1)$ from an element(s) indexed by $A'$. In either case there must exists an $l \in A'$ such that $x_{1l}+x_{2l}=1$. Therefore, $\displaystyle\sum_{j \in A}( x_{1j} + x_{2j}) \le |A|-1+ \displaystyle\sum_{j \not\in A}(x_{1j} + x_{2j})$.

Since $v$ satisfies all of the defining inequalities of $\Delta(m)$, we have $v\in \Delta(m)$. Containment of $K(m)$ in $\Delta(m)$ follows from the convexity of $K(m)$.
\end{proof}

The opposite inclusion, that $\Delta(m) \subseteq K(m)$, is not immediately clear. However, the following shows that the inclusion holds as long as $\Delta(m)$ is integral.

\begin{theorem}
\label{thm:intimpliesans}
If $\Delta(m)$ is integral, then $\Delta(m)=K(m)$.
\end{theorem}
\begin{proof}
Let $v$ be an integral vertex of $\Delta(m)$ which is not a vertex of $K(m)$. Using Definition~\protect\ref{thm:K3vert}, we can identify $v$ with a sequence of group elements $g_{1},\dots,g_{m}$. Since $v$ is not a vertex of $K(m)$, then $\displaystyle\sum_{k=1}^{m}g_{k}\neq (0,0)$. 

Where possible, we pair each nontrivial element of $g_{1},\dots,g_{m}$ with an identical group element. Since the sum is not the identity, there must be one remaining nontrivial element $g_{k}$ or a pair of distinct nontrivial elements $g_{k}$ and $g'_{k}$. Without loss of generality, we assume $g_{k}=(1,0)$ and, if an additional remaining element exists, that $g'_{k}=(0,1)$.

It follows that $A=\{l| g_{l}=(1,0)$ or $(1,1)\}$ has odd cardinality. By construction, we have $\displaystyle\sum_{j\in A}x_{1j}+x_{3j}=|A|$ and $\displaystyle\sum_{j\notin A}x_{1j}+x_{3j}=0$. This contradicts the hypothesis that $v$ was an element of $\Delta(m)$ since $v$ does not satisfy
\[\displaystyle\sum_{i\in A}(x_{1j}+x_{3j})\leq|A|-1+\displaystyle\sum_{j\notin A}(x_{1j}+x_{3j}).\] Therefore, all integral vertices of $\Delta(m)$ are also vertices of $K(m)$. By convexity this shows that if $\Delta(m)$ is integral, then it is a subset of $K(m)$. Combining this with Theorem~\protect\ref{thm:KinDelta} yields the equality of polytopes.
\end{proof}

We have thus exchanged the NP-complete problem of converting a V-representation to an H-representation with another NP-complete problem: recognizing when a rational polytope is integral \protect\cite{integralnphard}. However, the relationship between the $A$-inequalities of $\Delta(m)$ and those of the demihypercube suggest a change of coordinates that allow us to demonstrate integrality.

\section{An Alternative Description of the {Kimura-3} Polytope}
\label{sec:alternative}
\subsection{Coordinate Change from $\Delta (m)$ to $\Delta' (m)$}
\label{sec:deltaprime}
The $A$-inequalities for the demihypercube and for $\Delta(m)$ share the same underlying indexing structure. However, a more subtle connection with the demihypercube is revealed after a change of coordinates motivated by the group homomorphisms from $\bZ_2 \times \bZ_2 \rightarrow \bZ_2$. 

\begin{definition}
\label{def:map}
Define a map $f:\bR^{3m} \rightarrow \bR^{3m}$ by $f(M)=M'$, where $M'$ is a $3 \times m$ matrix with $x'_{1j}=x_{1j}+x_{2j}$, $x'_{2j}=x_{1j}+x_{3j}$ and $x'_{3j}=x_{2j}+x_{3j}$ for $1 \le j \le m$.
\end{definition}


\begin{definition}
\label{def:primes}
Define the polytope $\Delta'(m)=f(\Delta(m))$. 
\end{definition}

The function $f$ is not an isomorphism of the underlying integer lattices. However, the following computations show $\Delta(m)$ and $\Delta'(m)$ are isomorphic, and that they are either both integral or both non-integral polytopes.
\begin{lemma}
\label{thm:isomorphism}
Let $f_j:\bR^3 \rightarrow \bR^3$ be the restriction of $f$ to map from column $j$ of the matrix $M$ to column $j$ of the matrix $M'$. Then $f_j$ is an isomorphism which maps the unit 3-simplex in $\bR^3$ to $DH(3)$.
\end{lemma}

\begin{proof}
The vertices of the unit 3-simplex are \{(0,0,0),(1,0,0),(0,1,0),(0,0,1)\}. The map $f_j:\bR^3 \rightarrow \bR^3$ is given by $f_j(x_{ij})=(x_{1j}+x_{2j},x_{1j}+x_{3j},x_{2j}+x_{3j})$. The function $f_j$ is an isomorphism with inverse \[f_j^{-1}(x,y,z)=(\displaystyle\frac{x+y-z}{2},\displaystyle\frac{x+z-y}{2},\displaystyle\frac{y+z-x}{2}).\] The second part of the claim follows from applying the change of coordinates to the vertices of the 3-simplex.
\end{proof}

\begin{corollary} 
\label{cor:isomorphic} The polytopes $\Delta(m)$ and $\Delta'(m)$ are isomorphic. 
\end{corollary}
\begin{corollary} 
\label{cor:isomorphicint} The polytope $\Delta(m)$ is integral if and only if $\Delta'(m)$ is integral.
\end{corollary}
\begin{proof}
It is clear that if $\Delta(m)$ is integral then $\Delta'(m)$ must also be integral. Now assume $P \in \Delta(m)$ is a non-integral vertex which maps to an integral vertex of $\Delta'(m)$. This implies $x_{1j}+x_{2j} =1$ and $x_{1j},x_{2j} > 0$. It follows that $x_{1j}+x_{3j} = 1$ and $x_{2j}+x_{3j}=1$ as well. Combining these three equations yields $x_{1j}=x_{2j}=x_{3j}=\frac{1}{2}$. However this implies $x_{1j}+x_{2j}+x_{3j} > 1$, so $P$ could not have been an element of $\Delta(m)$.
\end{proof}

\begin{corollary}
The facets of $\Delta'(m)$ are given by: \[\sum_{j \in A}(x'_{i,j} ) \le |A|-1+ \sum_{j \not\in A}x'_{i,j} \] where $A$ is any subset of $\{1,2,\cdots m\}$ of odd cardinality and $1\le i \le 3$, and \[\sum_{i \in B}(x'_{i,j} ) \le |B|-1+ \sum_{i \not\in B}x'_{i,j} \] where $B$ is a subset of $\{1,2,3\}$ of odd cardinality and $1 \le j \le m$.  We call the former \emph{row facets} and the later \emph{column facets}.
\end{corollary}

\subsection{Psuedo-Demihypercubes}
\label{subsec:psuedo}
For a point $P$ in $\Delta'(m)$, we examine the connection between the number of integral coordinates in a row (resp. column) of $P$ and the number of row facets (resp. column facets) that $P$ lies on.

To explain this connection, we restrict our attention to the coordinates of a particular row of $m'$. Let $P|_r$ be the canonical projection of $P$ into $\bR^m$ corresponding to the coordinates of row $r$ of the matrix. We introduce the notion of a pseudo-demihypercube to describe the projection of $\Delta'(m)$ onto a particular row space.

\begin{definition}
For any row $r$, the \emph{psuedo-demihypercube} is defined by \[PDH(m) = \{P|_r:P \in \Delta'(m)\} \subseteq \bR^m.\] The hyperplanes corresponding to the row facets define \emph{psuedo-facets} of the polytope.
\end{definition}

We do not assume that $PDH(m)=DH(m)$ but we do make use of the observation that the coordinates of points in $PDH(m)$ must lie between zero and one.
\subsection{Integrality of Coordinates in the Psuedo-demihypercube}
\label{subsec:rowdemihypercube}
In this subsection we demonstrate a positive correlation between the number of integral coordinates of a point $P\in PDH(m)$ and the number of pseudo-facets $P$ lies on.

\begin{theorem}
\label{thm:twononinteger}
Every non-integral point $P$ of $PDH(m)$ has at least two non-integral coordinates. 
\end{theorem}

\begin{proof}
Let $P$ be a non-integral point of $PDH(m)$ with a single non-integral coordinate $p_k$, and let $I = \{i_1,i_2,\cdots,i_{l}\}$ denote the indices of coordinates of $P$ where $p_i=1$. The proof that there is a second non-integral coordinate is broken into cases based on the parity $|I|$. 

If $|I|$ is odd, then the $A$-inequality corresponding to $I$ is $|I| \le |I|-1+p_k$. It follows that $p_k \ge 1$, which contradicts the existence of a non-integral coordinate $P \in PDH(m)$.

If $|I|$ is even then let $A=I \cup \{i_k\}$. Then the $A$-inequality is $|I|+p_k \le |I|$ so $p_k \le 0$. This contradicts $p_k$ being a non-integral coordinate of $P \in PDH(m)$. Consequently $P$ cannot have exactly one non-integral coordinate.
\end{proof}

If $P$ has at at least two non-integral coordinates and lies on two psuedo-facets, then it cannot have any additional non-integral coordinates.

\begin{theorem}
\label{thm:twoplanes}
If a point $P$ lies on two pseudo-facets of $PDH(m)$, then $P$ has at most two  non-integral coordinates. 
\end{theorem}

\begin{proof}
Assume $P$ is a point in $PDH(m)$ which lies on facets $H_{A}$ and $H_{B}$ corresponding to index sets $A$ and $B$. We let $p_i$ denote the coordinates of $P$ indexed by $i \in \{1,2, \cdots, m \}$.

Then, 
\[ H_{A}: \displaystyle\sum_{i\in (A \backslash B)} p_i+\displaystyle\sum_{j\in (A \cap B)} p_j =|A \backslash B| +| A \cap B | -1 +\displaystyle\sum_{k\in (B \backslash A)} p_k+\displaystyle\sum_{l \in (A \cup B)'}p_l,\]
and
\[
H_{B}:\displaystyle\sum_{k\in (B \backslash A)} p_k+\displaystyle\sum_{j\in (A \cap B)} p_j =|B \backslash A| +| A \cap B | -1 +\displaystyle\sum_{i\in (A \backslash B)} p_i +\displaystyle\sum_{l \in (A \cup B)'}p_l.
\]
Combining these two conditions gives the following equation:
\[
2\displaystyle\sum_{j\in (A \cap B)} p_j=|A \backslash B|+| B \backslash A| +2| A \cap B| -2 +2\displaystyle\sum_{l \in (A \cup B)'}p_l.
\]
Since $A$ and $B$ are distinct sets of odd cardinality, it follows that $|A \backslash B| + |B \backslash A| \ge 2$.
In order for the above equation to hold for a point $P \in PDH(m)$, we must have $p_j=1$ for all $j\in A \cap B$, and $p_l=0$ for all $l \in (A \cup B)'$, and $|A \backslash B| + |B \backslash A|=2$. Consequently, there are at most two non-integral coordinates and they would have to be indexed by the two elements of $(A \backslash B) \cup (B \backslash A)$.
\end{proof}

If a point lies on three pseudo-facets, the correlation is stronger as all coordinates are forced to be integral. 

\begin{theorem}
\label{lem:threeplanes}
Let $P$ be a point in $PDH(m)$. If $P$ lies on three pseudo-facets, then $P$ is integral and lies on $m$ pseudo-facets.
\end{theorem}

\begin{proof}
Assume $P$ lies on pseudo-facets of $PDH(m)$ corresponding to odd cardinality sets $A$, $B$ and $C$. Repeated application of Theorem~\protect\ref{thm:twoplanes} yields: \[p_i=\left\{ \begin{array}{ll} 
1 & i \in \{A\cap B\} \cup \{A\cap C\} \cup \{B\cap C\}\\ 
0 & \mbox{otherwise} 
\end{array} \right\}. \]

 
Notice the integral point $P$ cannot have an odd number of coordinates with the value one or it would not satisfy the $A$-inequality where $I = \{i|p_i=1\}$. So, we may let $I = \{i_1,i_2,\cdots,i_{2k}\}$ denote the indices of coordinates of $P$ where $p_i=1$, and $J=\{j_{2k+1},\cdots,j_m\}$ denote the indices of coordinates where $p_j=0$.

The result follows from checking that $P$ lies on the $m$ psuedo-facets corresponding the sets $I \backslash \{i_l\}$ for $1 \le l \le 2k$, and $I \cup \{i_l\}$ for $2k+1 \le l \le m$.
\end{proof}

In summary, if a point in a pseudo-demihypercube lies on three or more psuedo-facets facets then it must be integral. If it lies on exactly two facets then it must have exactly two non-integral coordinates. Moreover, no point in the pseudo-demihypercube can have exactly one non-integral coordinate.

In addition to this structure, the results in this section demonstrate that the number of rows or columns which contain a non-integral coordinate is constrained by the total number of non-integral coordinates.  We introduce the following notation to make this precise.

\begin{definition}
Let $P$ be a point in $\Delta'(m)$. We define $k(P)$ as the number of non-integral coordinates in $P$ and $\omega(P)$ as the sum of the number of rows and columns of the matrix representation of $P$ which contain a non-integral coordinate.
\end{definition}

\begin{corollary}
\label{cor:cases28}
If $P \in \Delta'(m)$, then $k(P)\geq \omega(P)$ with equality met only when $P$ lies on exactly two row-facets (resp. column-facets) of each row (resp. column) containing a non-integral coordinate.
\end{corollary}

\subsection{Pseudo-Facet Classification}
\label{subsec:same}
In addition to knowing how the pseudo-facet structure constrains coordinate integrality, the proof of the integrality of $\Delta'(m)$ requires an additional classification of the pseudo-facets of $PDH(m)$.
\begin{definition}
\label{def:SOfaces}
Let $P$ be a point in $PDH(m)$ with exactly two non-integral coordinates, $p_{i}$ and $p_{j}$. Assume $P$ lies on a pseudo-facet $H$ corresponding to a set $A$. If $i$ and $j$ are both elements of $A$ or both elements of $A'$, we call $H$ a same facet or \emph{S-facet}. Otherwise, we call $H$ an opposite or \emph{O-facet}.
\end{definition}

This classification of facets is dependent on a choice of a point and does not universally classify facets into two disjoint sets. However, as we only apply the concept to the case when a point has been specified we suppress the unneeded notation that would indicate that the classification is function of a point $P$.

\begin{theorem}
\label{thm:parity}
Let $P$ be a point in $PDH(m)$ with exactly two non-integral coordinates, $p_{1}$ and $p_{2}$, which lies on a pseudo-facet $H$. Let $I=\{i| p_i=1\}$.  It follows that \[|I| \mbox{ is} \left\{ \begin{array}{ll} 
\mbox{odd} & \mbox{ if }H \mbox{ is a S-facet}\\ 
\mbox{even} & \mbox{ if }H \mbox{ is an O-facet.} 
\end{array} \right \}. \]
\end{theorem}

\begin{proof}
Let $P \in PDH(m)$ have exactly two non-integral coordinates, $p_1$ and $p_2$, that lie on a pseudo-facet $H$ with index set $A$.
\[H:\sum_{i \in A}p_{i} = |A|-1+ \sum_{j \not\in A}p_{j},\]

If $H$ is a S-facet, then to satisfy $H$, we must have $p_{1}+p_{2}\in\bZ$. Since ${\{p_1,p_2\} \in (0,1)}$, we must have $p_1+p_2=1$.

Now assume $\{p_1,p_2\}\in A$, then following equation represents $H$:
\[ p_{1}+p_{2}+|A \cap I|=| A| -1 +|A' \cap I|.\]
Since $|A \cap I| \le |A|-2$, this equation can only be satisfied when $|A \cap I|=|A|-2$ and $A' \cap I = \emptyset$. This means $|I|=|A|-2$,which is odd because $| A|$ is odd. 

If we assume $\{p_1,p_2\}\in A'$, we obtain the following equation representing H:
\[ |A \cap I|=| A| -1 + p_{1}+p_{2}+|A' \cap I|,\]
which reduces to 
\[ |A \cap I|=| A| +|A' \cap I|.\]
This equation can only be satisfied when $|I|=|A|$ which forces $|I|$ to be odd. 

If $H$ is an O-facet, then we must have $p_{1}-p_{2}\in\bZ$. Since $p_{1},p_{2}\in (0,1)$, this forces $p_{1}=p_{2}$. After canceling the non-integral coordinates, the equation for defining $H$ reduces to,
\[ |A \cap I|=| A| -1 +|A' \cap I|.\]
Since $|A \cap I| \le |A|-1$, the preceding equation can only be satisfied if $|A \cap I|=|A|-1$, and $A' \cap I= \emptyset$. This demonstrates that $|I|=|A|-1$, and is therefore even.
\end{proof}

\begin{corollary}
\label{lem:s-o}
If $P \in PDH(m)$ has exactly two non-integral coordinates, then $P$ cannot lie on both an O-facet and an S-facet. 
\end{corollary}

\section{Proof of the Facet Description for the {Kimura-3} Polytope}
\label{sec:integrality}

The properties of the pseudo-demihypercubes discussed in \S \protect\ref{sec:alternative} will be used to show that $\Delta'(m)$ is an integral polytope by demonstrating there is an open interval containing any non-integral point in $\Delta'(m)$. 

\begin{definition} A point $P$ is in the \emph{interior} of $\Delta'(m)$ if there exists a non-zero vector $v\in \bR^{3m}$ and an $\epsilon > 0$ such that $P+\lambda v \in \Delta'(m)$ whenever $\lambda \in (-\epsilon, \epsilon)$.
\end{definition}

The proof of integrality of $\Delta'(m)$ utilizes the following two lemmas which serve as tools for demonstrating that a non-integral point $P$ is in the interior of $\Delta'(m)$.

\begin{lemma}
\label{lem:compatibleinterval}
Let $P\in \Delta'(m)$. If the number of non-integral coordinates, $k(P)$, is greater than than $\omega(P)$, the sum of the number of rows and columns which contain non-integral coordinates, then $P$ is in the interior of $\Delta'(m)$. 
\end{lemma}

\begin{proof}
Let $P$ be a point in $\Delta'(m)$ such that $k(P)>\omega(P)$.
We construct a vector $v$, and $\epsilon > 0$ such that $P +\lambda v\in \Delta'(m)$ for all $\lambda \in (-\epsilon,\epsilon)$.

If $x'_{i,j}$ is an integral coordinate of $P$, then we set $v_{i,j}=0$. We setup a linear system of equations to solve for the remaining coordinates of $v$. For convenience, we linearly reorder the coordinates of $P$ such that $x_1,x_2,\cdots, x_k$ are non-integral.

Let $M$ be the $3 \times m$ matrix representation of $P$. For each row or column of $M$ containing non-integral coordinates, $P$ may lie on zero, one, or two of the associated facets. If $P$ does not lie on any facet, then there exists an $\epsilon$ such that $P+\lambda v$ will satisfy the corresponding inequalities for any choice of $v$.

If $P$ lies on a single facet $H$, then $H$ defines a single homogeneous linear relation on $v_1,v_2,\cdots, v_k$, since $P+v$ would satisfy the same linear relation so long as ${\displaystyle\sum_{j\in A} v_j - \displaystyle\sum_{j \not \in A} v_j = 0}$ where $A$ is the indexing set which defines the facet.

Additionally, if $P$ lies on two facets in a particular row or column, by the proof of Theorem~\protect\ref{thm:parity} the coordinates $v_1,\cdots,v_k$ must satisfy a single linear homogeneous relationship.  Explicitly we set $v_i=0$ if the $i$th coordinate is integral, and set the sum (or difference) of the $v$ coordinates equal to zero for the two indices corresponding to non-integral coordinates of $P$.

Therefore we get a system of up to $\omega(P)$ homogeneous linear equations in $k(P)$ unknowns. By the hypothesis $k(P) > \omega(P)$, so there exists a nontrivial solution for $v$. Since $v_{i,j}= 0$ whenever $P_{i,j}$ is integral we may choose an $\epsilon$ such that $\lambda v + P \in [0,1]^{3m}$ for all $\lambda \in (-\epsilon,\epsilon)$. It follows that $\lambda v + P \in \Delta'(m)$ for all $\lambda \in (-\epsilon,\epsilon)$.
\end{proof}

Lemma~\protect\ref{lem:compatibleinterval} is sufficient for constructing an interval for most non-integral points of $\Delta'(m)$. When $k(P)=\omega(P)$ we explicitly construct an interval containing $P$. Up to reordering of the rows and columns, only the two configurations of non-integral coordinates, shown in Figure~\protect\ref{2cases}, allow for $P$ be on exactly two row-facets and two column-facets for each row and column with a non-integral coordinate in $\Delta'(m)$. 
\begin{figure}
\[P_1=
\left(\begin{array}{cccccc}
\y x'_{1,1} & \y x'_{1,2} & x'_{1,3} & \cdots & x'_{1,m} \\
\y x'_{2,1} & \y x'_{2,2} & x'_{2,3} & \cdots & x'_{2,m} \\
x'_{3,1} & x'_{3,2} & x'_{3,3} &\cdots & x'_{3,m}
 \end{array}\right) P_2=
\left(\begin{array}{cccccc}
\y x'_{1,1} & \y x'_{1,2} & x'_{1,3} &\cdots & x'_{1,m} \\
\y x'_{2,1} & x'_{2,2} & \y x'_{2,3} & \cdots & x'_{2,m} \\
x'_{3,1} & \y x'_{3,2} & \y x'_{3,3} & \cdots & x'_{3,m}
 \end{array}\right)
\]
\caption{Configurations of non-integral coordinates for Lemma~\protect\ref{lem:intervalconstruction}. Coordinates in bold are the only non-integral coordinates of $P$.}
\label{2cases}
\end{figure}

\begin{lemma}
\label{lem:intervalconstruction}
If $P \in \Delta'(m)$ has one of the configurations of non-integral coordinates shown in Figure \ref{2cases}, then $P$ is in the interior of $\Delta'(m)$. 
\end{lemma}

\begin{proof}
Given a point $P \in \Delta'(m)$ with a configuration of non-integral coordinates as displayed in Figure~\protect\ref{2cases} we demonstrate that $P$ is in the interior of $\Delta'(m)$. 

In these two cases each row and column has at most two non-integral coordinates, so we use the S and O-facet description to assist in the construction. Let $|S|$ be the number of S-facets that $P$ lies on, and $|O|$ be the number of $O$-facets that $P$ lies on. 

To build the interval, we first note that for each configuration in Figure~\protect\ref{2cases} there exists a Hamiltonian cycle: \[(p_1,p_2,\cdots,p_k)=x'_{1,1}, x'_{2,1},\cdots x'_{1,1}\] in the graph with vertices corresponding to non-integral coordinates, and edges connecting non-integral coordinates in the same row or column. 

In the case of $P_1$ (resp. $P_2$) we reorder the coordinates of $v=(v_1,v_2,v_3,v_4, \cdots, v_{3m})$ (resp. $v=(v_1,\cdots v_6,\cdots, v_{3m})$) with the first four (resp. six) coordinates corresponding to the non-integral coordinates of $P_1$ (resp. $P_2$) in the order of the Hamiltonian cycle. 

Using the reordered coordinates we construct a vector $v$ as follows. First set $v_1=1$, and for $v_1$ through $v_4$ (resp. $v_6$) assign:
\[v_{i+1}=\left\{ \begin{array}{lr} 
v_i &\mbox{if } p_i \mbox{ and }p_{i+1} \mbox{ lie on an O-facet}\\ 
-v_i &\mbox{if } p_i \mbox{ and }p_{i+1} \mbox{ lie on an S-facet}. 
\end{array} \right. \]

We set $v_i=0$ for all remaining coordinates. Such a collection is consistent for the set of linear constraint defined in the proof of Lemma~\protect\ref{lem:compatibleinterval} applied to each consecutive pair of non-integral coordinates.

The constraint induced by $v_1$ and $v_4$ (resp. $v_6$) is only consistent if there are an even number of S-facets, otherwise $v_1$ would have to simultaneously be one and negative one

To show that there are an even number of S-facets we let $I$ denote the set of indices such that $x'_{i,j}=1$. We can compute $|I|$ by taking half of the number of coordinates with value one in each row, plus half of the number of coordinates with value one in each column.

Let $|O|$ denote the number of rows or column facets for which $P$ restricted to that row (resp. column) lies only on $O$-facets, and $|S|$ denote the rows or column facets for which $P$ restricted to that row (resp. column) lies only on $S$-facets. By Theorem~\protect\ref{thm:parity} every S-column and S-row contains an odd number of coordinates with value one, while every O-row and O-column must contain an even number of coordinates with value one. Therefore since $|I|=\frac{1}{2}((2q_1+1)|S|+2q_2|O|)$ is an integer, we know that the number of S-facets must be even. This confirms the consistency of the vector $v$.

We now choose an $\epsilon$ which restricts $\lambda v + P$ to $[0,1]^{3m}$. Then it follows from Theorem~\protect\ref{thm:parity} that $P + \lambda v \in \Delta'(m)$ for all $\lambda \in (-\epsilon,\epsilon)$. Thus $P$ is in the interior of $\Delta'(m)$.
\end{proof}

We utilize these lemmas to prove that $\Delta'(m)$ is integral. 

\begin{theorem}
\label{thm:Deltaprimeintegral}
The polytope $\Delta'(m)$ is integral. 
\end{theorem}

\begin{proof}
Let $P$ be a non-integral point in $\Delta'(m)$. By Theorem~\protect\ref{lem:threeplanes} there exists a row or column for which $P$ lies on two or fewer row or column facets.

If every non-integral coordinate is on a row and column for which $P$ lies on exactly two facets, then by Theorem~\protect\ref{thm:twoplanes} there are exactly two non-integral coordinates in each such row and column. By reordering the coordinates we can thus assume $P$ has one of the configurations in Figure~\protect\ref{2cases}. Then, by Lemma~\protect\ref{lem:intervalconstruction}, $P$ is is an interior point of $\Delta'(m)$. 

Otherwise, there must exist a row or column for which $P$ lies on fewer than two facets and thus has more than two non-integral coordinates in that row or column. This ensures that $k(P)> \omega(P)$, and thus by Lemma~\protect\ref{lem:compatibleinterval} $P$ is an interior point of $\Delta'(m)$. Therefore every non-integral point of $\Delta'(m)$ lies in the interior.
\end{proof}

\begin{theorem}
\label{thm:Deltaintegral}
The polytope $\Delta(m)$ is an H-representation of $K(m)$.
\end{theorem}

\begin{proof}
By Theorem~\protect\ref{thm:Deltaprimeintegral} the polytope $\Delta'(m)$ is integral. Applying Corollary~\protect\ref{cor:isomorphicint} shows $\Delta(m)$ is also integral. Finally, by Theorem~\protect\ref{thm:intimpliesans} we have $K(m)=\Delta(m)$.
\end{proof}

\section{Conclusion}
\label{sec:conclusion}
Phylogenetic varieties have been used to answer identifiability questions \protect\cite{twotreemixture,threetreemixture} and to develop tree reconstruction algorithms \protect\cite{eriksvd, invnew,ibqp}. Greater understanding of the geometry of phylogenetic varieties, including an understanding of the singularity locus, has improved the speed and accuracy of these reconstruction methods \protect\cite{relevant,jcineq}. Additionally, in light of recent connections with conformal blocks and Berenstein-Zelevinsky triangles, a deeper understanding of phylogenetic varieties is also of interest to algebraic geometers \protect\cite{cblocksandBZ}. 

The H-representation provides a new vantage point for understanding the {Kimura-3} variety. The authors hope this will lead to a better understanding of the geometry and associated biology of the {Kimura-3} variety, and of group based models in general. To this end, we describe open problems of both mathematical and biological interest. We begin with a fundamental question about the geometry of the {Kimura-3} variety.
\begin{op}
Classify the singularity structure of the {Kimura-3} varieties. (see \protect\cite{binary} for an analogous study in the binary symmetric case)
\end{op}

 In the binary symmetric case, the variety associated to any tree with $m$ leaves is deformation equivalent to the variety associated to the $m$-claw tree \protect\cite{binary}. While such a relationship does not hold in the {Kimura-3} case \protect\cite{hilbert}, one would still like to understand the geometric relationship among varieties associated to different $n$ leaf trees. From a biological perspective this problem can be posed as follows:

\begin{op}
Describe the geometric relationship between two {Kimura-3} varieties whose trees differ by a single nearest neighbor interchange.
\end{op}

These geometric questions are closely related to the combinatorics of the polytopes themselves. We hope that the H-representation will help provide an answer to the following combinatorial problem:

\begin{op}
Compute the $f$-vector and Hilbert polynomial of the polytope associated to the {Kimura-3} model for the $m$-claw tree. 
\end{op}

\begin{op}
Describe the H-representation for the polytope associated to the $m$-claw tree for an arbitrary finite abelian group.
\end{op}

\section{Acknowledgements}
This material is based upon work supported by the National Science Foundation under Grant No. DMS-1358534. Research reported in this publication was supported by an Institutional Development Award (IDeA) from the National Center for Research Resources (5 P20 RR016461) and the National Institute of General Medical Sciences (8 P20 GM103499) from the National Institutes of Health.

The authors would like to thank Valery Alexeev, Wayne Anderson, Kaie Kubjas and Mateusz Michalek for their comments and suggestions during the development of this work.

\bibliographystyle{siam}

\bibliography{sample}

\begin{thebibliography}{10}

\bibitem{twotreemixture}
{\sc Elizabeth~S Allman, Sonia Petrovic, John~A Rhodes, and Seth Sullivant},
  {\em Identifiability of two-tree mixtures for group-based models}, IEEE/ACM
  Transactions on Computational Biology and Bioinformatics (TCBB), 8 (2011),
  pp.~710--722.

\bibitem{binary}
{\sc Weronika Buczynska and Jaroslaw~A Wisniewski}, {\em On the geometry of
  binary symmetric models of phylogenetic trees}, Journal of the European
  Mathematical Society, 9 (2007), pp.~609--635.

\bibitem{geometryofk3}
{\sc Marta Casanellas and Jes{\'u}s Fern{\'a}ndez-S{\'a}nchez}, {\em Geometry
  of the {Kimura} 3-parameter model}, Advances in Applied Mathematics, 41
  (2008), pp.~265--292.

\bibitem{relevant}
\leavevmode\vrule height 2pt depth -1.6pt width 23pt, {\em Relevant
  phylogenetic invariants of evolutionary models}, Journal de Math{\'e}matiques
  Pures et Appliqu{\'e}es, 96 (2011), pp.~207--229.

\bibitem{local}
{\sc Marta Casanellas, Jes{\'u}s Fern{\'a}ndez-S{\'a}nchez, and Mateusz
  Micha{\l}ek}, {\em Local description of phylogenetic group-based models},
  arXiv preprint arXiv:1402.6945,  (2014).

\bibitem{cavender}
{\sc J.~Cavender and J.~Felsenstein}, {\em Invariants of phylogenies in a
  simple case with discrete states}, J Classificiation, 4 (1987), pp.~57--71.

\bibitem{cox}
{\sc David~A. Cox, John~B. Little, and Henry~K. Schenck}, {\em Toric
  varieties}, American Mathematical Soc., 2011.

\bibitem{dbm}
{\sc Maria Donten-Bury and Mateusz Michalek}, {\em Phylogenetic invariants for
  group-based models}, arXiv preprint arXiv:1011.3236,  (2010).

\bibitem{eriksvd}
{\sc Nicholas Eriksson}, {\em Tree construction using singular value
  decomposition}, Algebraic Statistics for computational biology,  (2005),
  pp.~347--358.

\bibitem{phyloAG}
{\sc Nicholas Eriksson, Kristian Ranestad, Bernd Sturmfels, and Seth
  Sullivant}, {\em Phylogenetic algebraic geometry}, Projective Varieties with
  Unexpected Properties,  (2005), pp.~237--255.

\bibitem{evans}
{\sc Steven~N Evans and TP~Speed}, {\em Invariants of some probability models
  used in phylogenetic inference}, The Annals of Statistics,  (1993),
  pp.~355--377.

\bibitem{invnew}
{\sc Jes{\'u}s Fern{\'a}ndez-S{\'a}nchez and Marta Casanellas}, {\em Invariant
  versus classical quartet inference when evolution is heterogeneous across
  sites and lineages}, arXiv preprint arXiv:1405.6546,  (2014).

\bibitem{fulton}
{\sc William Fulton}, {\em Introduction to toric varieties}, no.~131, Princeton
  University Press, 1993.

\bibitem{verticesnphard}
{\sc Leonid Khachiyan, Endre Boros, Konrad Borys, Khaled Elbassioni, and
  Vladimir Gurvich}, {\em Generating all vertices of a polyhedron is hard},
  Discrete \& Computational Geometry, 39 (2008), pp.~174--190.

\bibitem{kimura}
{\sc Motoo Kimura}, {\em A simple method for estimating evolutionary rates of
  base substitutions through comparative studies of nucleotide sequences},
  Journal of Molecular Evolution, 16 (1980), pp.~111--120.

\bibitem{hilbert}
{\sc Kaie Kubjas}, {\em Hilbert polynomial of the {Kimura} 3-parameter model},
  arXiv preprint arXiv:1007.3164,  (2010).

\bibitem{cblocksandBZ}
{\sc Kaie Kubjas and Christopher Manon}, {\em Conformal blocks,
  {Berenstein--Zelevinsky} triangles, and group-based models}, Journal of
  Algebraic Combinatorics,  (2013), pp.~1--26.

\bibitem{threetreemixture}
{\sc Colby Long and Seth Sullivant}, {\em Identifiability of 3-class
  jukes--cantor mixtures}, Advances in Applied Mathematics, 64 (2015),
  pp.~89--110.

\bibitem{conformal}
{\sc Christopher~A Manon}, {\em The algebra of conformal blocks}, arXiv
  preprint arXiv:0910.0577,  (2009).

\bibitem{jcineq}
{\sc Frederick~A Matsen}, {\em Fourier transform inequalities for phylogenetic
  trees}, IEEE/ACM Transactions on Computational Biology and Bioinformatics
  (TCBB), 6 (2009), pp.~89--95.

\bibitem{geometryk3anytree}
{\sc Mateusz Micha{\l}ek}, {\em Toric geometry of the 3-{Kimura} model for any
  tree}, Advances in Geometry, 14 (2014), pp.~11--30.

\bibitem{pachtersturmfels}
{\sc Lior Pachter and Bernd Sturmfels}, {\em Algebraic statistics for
  computational biology}, vol.~13, Cambridge University Press, 2005.

\bibitem{integralnphard}
{\sc Christos~H. Papadimitriou and Mihalis Yannakakis}, {\em On recognizing
  integer polyhedra}, Combinatorica, 10 (1990), pp.~107--109.

\bibitem{ibqp}
{\sc Joseph~P Rusinko and Brian Hipp}, {\em Invariant based quartet puzzling.},
  Algorithms for Molecular Biology, 7 (2012), p.~35.

\bibitem{toricideal}
{\sc Bernd Sturmfels and Seth Sullivant}, {\em Toric ideals of phylogenetic
  invariants}, Journal of Computational Biology, 12 (2005), pp.~204--228.

\bibitem{toricfiber}
{\sc Seth Sullivant}, {\em Toric fiber products}, Journal of Algebra, 316
  (2007), pp.~560--577.

\end{thebibliography}

\end{document}